\documentclass[2pt,a4paper]{article}
\usepackage{amsmath}
\usepackage{amsfonts}
\usepackage{amsthm}
\usepackage{amssymb}
\usepackage{booktabs,amsmath}

\def\XXint#1#2#3{{\setbox0=\hbox{$#1{#2#3}{\int}$}
     \vcenter{\hbox{$#2#3$}}\kern-.5\wd0}}

\usepackage[english]{babel}
\usepackage{xcolor}
\usepackage{enumerate}

\theoremstyle{plain}
\newtheorem{lemma}{Lemma}[section]
        \newtheorem{proposition}[lemma]{Proposition}
        \newtheorem{theorem}[lemma]{Theorem}
        \newtheorem{definition}{Definition}[section]
        \newtheorem{remark}[lemma]{Remark}

\theoremstyle{definition}
\usepackage{tabularx}

\usepackage{amsmath,amsfonts,amssymb,amscd,lscape,amstext,amsthm}

\numberwithin{equation}{section}

{\setlength{\parindent}{0pt}
\title{ {\bf Determining an anisotropic conductivity by boundary measurements: stability at the boundary}}

\oddsidemargin 0.0in \evensidemargin 0.0in
{\setlength{\parindent}{0pt}
\author{Giovanni Alessandrini\thanks{Dipartimento di Matematica e Geoscienze, Universit\`{a} di Trieste, Italy. Email:alessang@units.it.}\qquad Romina Gaburro\thanks{Department of Mathematics and Statistics, University of Limerick, Ireland.  Email: romina.gaburro@ul.ie}\qquad\\ Eva Sincich\thanks{Dipartimento di Matematica e Geoscienze, Universit\`{a} di Trieste, Italy. Email:esincich@units.it.}}

\author{Giovanni Alessandrini\thanks{ Dipartimento di Matematica e Geoscienze, Universit\`{a} di Trieste, Italy. E-mail: \textsf{alessang@units.it}} \ \ \ \ Romina Gaburro\thanks{Department of Mathematics and Statistics, CONFIRM-Science Foundation Ireland, Health Research Institute (HRI), University of Limerick, Ireland.   E-mail: \textsf{romina.gaburro@ul.ie}}  \ \ \ \ Eva
Sincich\thanks{Dipartimento di Matematica e Geoscienze,
Universit\`a degli Studi di Trieste, Italy. E-mail: \textsf{esincich@units.it}} 
}

\linespread{1.6}

\date{}

\setlength{\parindent}{0pt}

\begin{document}

\maketitle

\noindent \textbf{Abstract:} 
We consider the inverse problem of determining, the possibly anisotropic, conductivity of a body by means of the so called local Neumann to Dirichlet map on a curved portion $\Sigma$ of the boundary. Motivated by the uniqueness result for piecewise constant anisotropic conductivities proved in \cite{Al-dH-G}, we provide a H\"{o}lder stability estimate on $\Sigma$ when the conductivity is a priori known to be a constant matrix near $\Sigma$.

\medskip

\medskip
 

\medskip

\medskip



\section{Introduction}\label{sec1}
\setcounter{equation}{0}
Given $\varphi :\partial\Omega\longrightarrow\mathbb{R}$, with zero average, consider the Neumann problem
\begin{displaymath}
\left\{ \begin{array}{ll} \textnormal{div}(\sigma\nabla
u)=0, &
\textrm{$\textnormal{in}\quad\Omega$},\\
\sigma\nabla u\cdot\nu = \varphi ,& \textrm{$\textnormal{on}\quad{\partial\Omega},$}
\end{array} \right.
\end{displaymath}
where $\sigma = \left\{\sigma_{ij}(x)\right\}_{i,j=1}^n$, with $x\in\Omega$, satisfies the uniform ellipticity condition
\begin{eqnarray}\label{ellitticita'sigma}
\lambda^{-1}\vert\xi\vert^{2}\leq{\sigma}(x)\xi\cdot\xi\leq\lambda\vert\xi\vert^{2}, \ for\:almost\:every\:x\in\Omega,\nonumber \ for\:every\:\xi\in\mathbb{R}^{n},
\end{eqnarray}
for some positive constant $\lambda$. Here, a standard, variational, functional framework is understood. Details will be presented in what follows. Electrical Impedance Tomography (EIT) is the inverse problem of determining the conductivity $\sigma$ when the Neumann-to-Dirichlet (N-D) map
\[\mathcal{N}_{\sigma}: \varphi\longrightarrow u\big\vert_{\partial\Omega}\]
is given,\cite{C}. It is well-known that if $\sigma$ is allowed to be anisotropic, i.e. a full matrix, although symmetric, then it is not uniquely determined by $\mathcal{N}_{\sigma}$. In fact, if $\Phi :\overline\Omega\longrightarrow\overline\Omega$ is a diffeomorphism such that $\Phi\big\vert_{\partial\Omega} = I$, then $\sigma$ and its push-forward under $\Phi$,
\[\Phi^{*}\sigma = \frac{(D\Phi )\sigma(D\Phi)^T}{\det(D\Phi)}\circ \Phi^{-1}\]
give rise to the same N-D map. This construction is due to Tartar, as reported by Kohn and Vogelius \cite{Ko-V1}. A prominent line of research investigates the determination of $\sigma$ 
modulo diffeomorphism which fix the boundary, in this respect we refer to the seminal paper of Lee and Uhlmann \cite{Le-U}. From another point of view, anisotropy cannot be neglected in applications, such as medical imaging or geophysics. It is therefore interesting to investigate possible kinds of structural assumptions (physically motivated) under which unique determination of $\sigma$ from $\mathcal{N}_{\sigma}$ is restored.

In \cite{Al-dH-G} the case of a piecewise constant conductivity $\sigma$ was treated, and assuming that the interfaces of discontinuity contain portions of curved (non flat) hypersurfaces, uniqueness was proven. Subsequently \cite{Al-dH-G-S}, uniqueness was proven also in cases of a layered structure with unknown interfaces. We refer to these two papers for a bibliography on the relevance of anisotropy in applications.

It is still an open problem, to prove, in such settings, estimates of stability. Indeed, a line of research initiated by Alessandrini and Vessella \cite{Al-V} in the isotropic case (i.e.: $\sigma = \gamma I$, with $\gamma$ scalar), suggests that also in the setting of Alessandrini-de Hoop-Gaburro \cite{Al-dH-G} a Lipschitz stability estimate might hold. Such a generalization, however, does not appear to be an easy task, because isotropy intervenes in many steps of the proof in \cite{Al-V}.

In this note we treat the first step in a program to prove stability for piecewise constant anisotropic conductivity with curved interfaces. More precisely, assuming $\sigma$ constant in a neighborhood $\mathcal{U}$ of a curved portion $\Sigma$ of the boundary $\partial\Omega$, we show that $\sigma\big\vert_{\mathcal{U}}$ depends in a H\"older fashion on the local N-D map $\mathcal{N}_{\sigma}^{\Sigma}$ (a precise definition will be given in due course).

Since Kohn and Vogelius \cite{Ko-V}, \cite{Ko-V2} and Alessandrini \cite{A}, \cite{Al}, it is customary to treat the uniqueness and stability, at the boundary, as a first step towards determination in the interior. And also in the anisotropic case we wish to mention the results of Kang and Yun \cite{K-Y} who proved reconstruction and stability at the boundary  up to diffeomorphisms which keep the boundary fixed.

Here, assuming a quantitative formulation of non-flatness of $\Sigma$, we are able to stably determine the full conductivity matrix $\sigma$ (near $\Sigma$). More precisely, we shall assume that there exists three points $P_1, P_2, P_3\in\Sigma$ such that the corresponding unit normal vectors to $\Sigma$, $\nu(P_1), \nu(P_2), \nu(P_3)$ are quantitatively pairwise distinct. 

Our argument here is based on various new features. As noticed in \cite{Al-dH-G}, the local N-D map $\mathcal{N}_{\sigma}^{\Sigma}$ is an integral operator whose kernel $K$ differs from the well-known Neumann kernel $N$ by a bounded correction term.

The common feature of the two kernels is the character of their singularity which in turn encapsulates information on the tangential part of the metric $\left\{g_{ij}\right\}_{i,j = 1}^n$ associated to the conductivity $\sigma$
\begin{equation}\label{metrica}
g_{ij} = \left(\det\sigma\right)^{\frac{1}{n-2}} (\sigma^{-1}_{ij})
\end{equation}
in dimension $n\geq 3$. By testing $\mathcal{N}_{\sigma}^{\Sigma}$ on suitable combination of mollified $\delta$ functions, we achieve a quantitative evaluation of the tangential component of $\left\{g_{ij}\right\}$. Next, by exploiting the quantitative notion of `non-flatness' of $\Sigma$, we show that the full metric $\left\{g_{ij}\right\}$ can be recovered from three tangential samples at three different points with sufficiently distinct tangent planes, or, equivalently, pairwise distinct normal vectors.\\The paper is organized as follows. Section 2 contains the main definitions, including the quantified notion of non-flatness of $\Sigma$ where we localize the measurements $\mathcal{N}_{\sigma}^{\Sigma}$ (Section \ref{Notation and assumptions}). This section also contains the statement of our main result of stability at the boundary of anisotropic conductivities $\sigma$ that are constant near $\Sigma$ in terms of $\mathcal{N}_{\sigma}^{\Sigma}$ (Section \ref{local stability main result}).
Section \ref{mollifiers on graphs} is devoted to the construction of mollified delta functions on $\Sigma$. The proof of the main result is a two steps procedure. In the first step, contained in Section \ref{stability tangential g}, we stably  recover the tangential component of $g$ in terms of $\mathcal{N}_{\sigma}^{\Sigma}$ . The argument of this proof is based on estimating the asymptotic behaviour of the Neumann Kernel $N(\cdot, y)$ of $L=\mbox{div}(\sigma\nabla\cdot)$  and its derivative, near the pole $y\in\Sigma$, and the sampling $\mathcal{N}_{\sigma}^{\Sigma}$ on suitable combinations of the mollifiers  given in Section \ref{mollifiers on graphs}. In the second step, discussed in Section \ref{stability full g}, exploiting the non flatness condition on $\Sigma$ as well as the structure of the metric $g$, we derive the stability on the boundary for the full metric $g$ which in turn leads to the stable determination of the conductivity $\sigma$ on $\Sigma$ .



\section{Main Result}\label{sec2}
\setcounter{equation}{0}
\subsection{Notation and assumptions}\label{Notation and assumptions}
In several places in this manuscript it will be useful to single out one coordinate direction. To this purpose, the following notations for points $x\in \mathbb{R}^n$ will be adopted. For $n\geq 3$, a point $x\in \mathbb{R}^n$ will be denoted by $x=(x',x_n)$, where $x'\in\mathbb{R}^{n-1}$ and $x_n\in\mathbb{R}$. Moreover, given a point $x\in \mathbb{R}^n$ and given $a,b\in \mathbb{R}$, we shall denote with $B_r(x)$, $B_r(x')$ the open balls in $\mathbb{R}^{n}, \mathbb{R}^{n-1}$ centred at $x$, $x'$, respectively, with radius $r$ and by $ Q_{a,b}(x)$ the cylinder $\{y=(y',y_n)\in\mathbb{R}^n \ : \  \large| x'-y'\large|<a \ ; \  |x_n-y_n|<b\}$. It will also be understood that $B_r=B_r(0)$ , $B^{'}_r = B_r(0')$ and $Q_{a,b}=Q_{a,b}(0)$.\\

We shall assume throughout that $\Omega\subset\mathbb{R}^n$, with $n\geq 3$, is a bounded domain with Lipschitz boundary, as per definition \ref{def lipschitz} below.

\begin{definition}\label{def lipschitz}
We will say that $\partial\Omega$ is of Lipschitz class with constants $r_0, L>0$, if for every $P\in\partial\Omega$, there exists a rigid transformation of coordinates under which we have $P=0$ and

\[\Omega\cap Q_{r_0,r_0L} = \left\{(x',x_n)\in Q_{r_0,r_0L}\:|\: x_n> \varphi(x') \right\},\]

where $\varphi$ is a Lipschitz continuous function on $B'_{r_0}$ satisfying
\[\varphi(0) = 0\]
and
\[ |\varphi(x')- \varphi(y')|\le L |x'-y'| , \ \ \mbox{for every}\ \  x',y' \in B'_{r_0} .\]


\end{definition}

We fix an open non-empty subset $\Sigma$ of $\partial\Omega$ (where the measurements in terms of the local N-D map are taken). A precise definition of the N-D map and its local version with respect to $\Sigma$ are given below.


\subsubsection{The Neumann-to-Dirichlet map.}\label{N-to-D}


 
Denoting by $Sym_n$ the class of $n\times n$ symmetric real valued matrices, we assume
that $\sigma\in L^{\infty}(\Omega\:,Sym_{n})$ satisfies the ellipticity condition \eqref{ellitticita'sigma}.We consider the function spaces
\begin{equation*}
_{0}H^{\frac{1}{2}}(\partial \Omega)=\left\{f\in
H^{\frac{1}{2}}(\partial \Omega)\vert\:\int_{\partial\Omega}f\:
=0\right\},
\end{equation*}
\begin{equation*}
_{0}H^{-\frac{1}{2}}(\partial \Omega)=\left\{\psi\in
H^{-\frac{1}{2}}(\partial \Omega)\vert\:\langle\psi,\:1\rangle=0
\right\}.
\end{equation*}
The global Neumann-to-Dirichlet map is then defined as follows.

\begin{definition}\label{definition N-D}
The Neumann-to-Dirichlet (N-D) map associated with $\sigma$,
\[\mathcal{N}_{\sigma}:\  _{0}H^{-\frac{1}{2}}(\partial \Omega)\longrightarrow \:_{0}H^{\frac{1}{2}}(\partial \Omega)\]
is characterized as the selfadjoint operator satisfying
\begin{equation}\label{ND globale}
\langle\psi,\:\mathcal{N}_{\sigma}\psi\rangle\:=\:\int_{\:\Omega} \sigma(x)
\nabla{u}(x)\cdot\nabla{u}(x)\:dx,
\end{equation}
for every $\psi\in\: _{0}H^{-\frac{1}{2}}(\partial \Omega)$, where 
$u\in{H}^{1}(\Omega)$ is the weak solution to the Neumann problem
\begin{equation}\label{N bvp}
\left\{ \begin{array}{lll}\displaystyle\textnormal{div}(\sigma\nabla u)=0, &
\textrm{$\textnormal{in}\quad\Omega$},\\
\displaystyle\sigma\nabla u\cdot\nu\vert_{\partial\Omega}=\psi, &
\textrm{$\textnormal{on}\quad{\partial\Omega}$},\\
\displaystyle\int_{\partial\Omega}u\: =0
\end{array} \right.
\end{equation}
and $\langle\cdot,\cdot\rangle$ denotes the $L^{2}(\partial\Omega)$-pairing between $H^{\frac{1}{2}}(\partial\Omega)$ and its dual $H^{-\frac{1}{2}}(\partial\Omega)$.
\end{definition}


For the local version of the N-D map, we consider an open portion of $\partial\Omega$, $\Sigma$, and, denoting by
$\Delta=\partial\Omega\setminus\overline\Sigma$, we introduce the subspace of $H^{\frac{1}{2}}(\partial \Omega)$,
\[H^{\frac{1}{2}}_{co}(\Delta)=\left\{f\in H^{\frac{1}{2}}(\partial \Omega)\:|\: \mbox{supp}(f)\subset\Delta\right\}\]
together with its closure in $H^{\frac{1}{2}}(\partial\Omega)$, $H^{\frac{1}{2}}_{00}(\Delta)$. We also introduce
\begin{equation}
_{0}H^{-\frac{1}{2}}(\Sigma)=\left\{\psi\in \:
_{0}H^{-\frac{1}{2}}(\partial\Omega)\vert\:\langle\psi,\:f\rangle=0,\quad\textnormal{for\:any}\:f\in
H^{\frac{1}{2}}_{00}(\Delta)\right\},
\end{equation}
that is the space of distributions $\psi \in
H^{-\frac{1}{2}}(\partial\Omega)$ which are supported in
$\overline\Sigma$ and have zero average on $\partial\Omega$. The local
N-D map is then defined as follows.
\begin{definition}
The local Neumann-to-Dirichlet map associated with $\sigma$,
$\Sigma$ is the operator $\mathcal{N}_{\sigma}^{\Sigma}:\:
_{0}H^{-\frac{1}{2}}(\Sigma)\longrightarrow
\big(_{0}H^{-\frac{1}{2}}(\Sigma)\big)^{\ast}\subset{_{0}H}^{\frac{1}{2}}(\partial\Omega)$
given  by
\begin{equation}
\langle \mathcal{N}_{\sigma}^{\Sigma}\;\varphi,\;\psi\rangle=\langle
\mathcal{N}_{\sigma}\;\varphi,\;\psi\rangle,
\end{equation}
for every $\varphi, \psi\in\:_{0}H^{-\frac{1}{2}}(\Sigma)$.
\end{definition}
Given $\sigma^{(i)}\in L^{\infty}(\Omega\:,Sym_{n})$, satisfying \eqref{ellitticita'sigma}, for $i=1,2$, the following equality holds true.
\begin{equation}\label{Alessandrini identity local N-D}
\left<\psi_1,\left(\mathcal{N}_{\sigma^{(2)}}^{\Sigma} - \mathcal{N}_{\sigma^{(1)}}^{\Sigma}\right)\psi_2\right> = \int_{\Omega} \left(\sigma^{(1)}(x) - \sigma^{(2)}(x)\right)\nabla u_1(x)\cdot\nabla u_2(x),
\end{equation}
for any $\psi_i\in\: _{0}H^{-\frac{1}{2}}(\Sigma)$, for $i=1,2$ and $u_i\in H^{1}(\Omega)$ being the unique weak solution to the Neumann problem
\begin{equation}
\left\{ \begin{array}{lll}\displaystyle\textnormal{div}(\sigma^{(i)}\nabla u_i)=0, &
\textrm{$\textnormal{in}\quad\Omega$},\\
\displaystyle\sigma^{(i)}\nabla u_i\cdot\nu\vert_{\partial\Omega}=\psi_i, &
\textrm{$\textnormal{on}\quad{\partial\Omega}$},\\
\displaystyle\int_{\partial\Omega}u_i\: =0.
\end{array} \right.
\end{equation}

\subsubsection{Non-flatness of $\Sigma$\label{non-flatness}}

\begin{definition}\label{def C2 alpha boundary}
Let $\Omega$ $\subset\mathbb{R}^n$ be as above. Given $\alpha$, $\alpha\in(0,1)$, we say that a portion $\Sigma$ of $\partial\Omega$ is of class $C^{2,\alpha}$ with constants $\rho$, $M>0$ if ,up to a rigid transformation of coordinates, $\Sigma$ is an $(n-1)$ dimensional $C^{2,\alpha}$ manifold with chart $({B'_{\rho}}, \varphi)$, where $\varphi: B'_{\rho}\subset\mathbb{R}^{n-1}\rightarrow\mathbb{R}$ is such that
\begin{equation}\label{norm phi}
 {\varphi(0) =  \large|\nabla\varphi(0)\large|}=0 \qquad  \|D^2 \varphi(x') - D^2 \varphi(y')\| \le M |x'-y'|^{\alpha} \ \ \mbox{for all }\  x',y'\in B'_{\rho} .
\end{equation}
We will also write
\begin{equation}\label{Sigma manifold}
\Sigma=\left\{(x', \varphi(x'))\:|\: |x'|\leq {\rho}\right\}.
\end{equation}



\end{definition}

For any $P\in\partial\Omega$, we will denote by $\nu(P)$ the outer unit normal to $\partial\Omega$ at $P$. 

\begin{definition}\label{definition non-flatness}
Given $\Sigma$ as above, we shall say that such a portion of a surface is non-flat (and equivalently the function $\varphi$) if, there exist three points $P_1, P_2, P_3\in\Sigma$ and a constant $C_0$, $0<C_0<1$, such that
\begin{equation}\label{nu1}
\nu(P_1)\cdot\nu(P_2) \le 1-C_0,
\end{equation}
\begin{equation}\label{nu2}
 \nu(P_1)\cdot\nu(P_3)\le 1-C_0,
 \end{equation}
\begin{equation}\label{nu3}
 \nu(P_2)\cdot\nu(P_3) \le 1-C_0 \ .
\end{equation}

\end{definition}


\subsection{Local stability at the boundary}\label{local stability main result}

It will be convenient to define  throughout this paper the following quantity
\begin{equation}\label{E}
E:=||\sigma^{(1)}-\sigma^{(2)}||_{L^{\infty}(\Omega)}.
\end{equation}

We will assume that there is a point $y\in\partial\Omega$ such that, up to a rigid transformation, $y=0$, and

\begin{equation}\label{definizione Sigma}
\Sigma = \partial\Omega\cap B_{\rho}
\end{equation}

is a non-flat portion of $\partial\Omega$ of class $C^{2,\alpha}$ with constants $\rho>0$, $M>0$ and $C_0>0$ as per definitions \ref{def C2 alpha boundary}, \ref{definition non-flatness}.

\begin{definition}
The set of parameters   $ \{\lambda, r_0, \:L, \:\rho,\:M,\: C_0,\:n\}$
 is called the a-priori data.
\end{definition}
The following notation will also be adopted throughout the manuscript.
\begin{enumerate}[i)]
\item A constant C is said to be uniform if it depends on the \textit{a-priori} data only.
\item We denote by $\mathcal{O}(t)$ a function $g$ such that
\begin{equation}\label{def O}
|g(t)|\leq Ct,\qquad\textnormal{for\:all}\:t,\quad 0<t<t_0,
\end{equation}
where $C, t_0>0$ are uniform constants.
\item We set, for $x'\in\mathbb{R}^{n-1}$,
\begin{displaymath}
\delta(x')=\left\{ \begin{array}{ll} C e^{\frac{1}{(|x'|^2-1)}} &
\textrm{$\textnormal{if} \quad |x'|< 1$},\\
0 & \textrm{$\textnormal{if}\quad{|x'|\ge 1},$}
\end{array} \right.
\end{displaymath}
where $C>0$ is a constant such that $\int_{\mathbb{R}^{n-1}} \delta(x')dx' = 1$.
\end{enumerate}

 In what follows, $\| \cdot\|_{\mathcal{L}({\mathcal{B}}_1, {\mathcal{B}}_2)}$ will denote the operator norm for linear operators between Banach spaces ${\mathcal{B}}_1, {\mathcal{B}}_2$.

Our main result is stated below.

\begin{theorem}\label{teorema principale}
Let $y$, $\rho$ and $\Sigma$ be defined by \eqref{definizione Sigma}. Let $\sigma^{(i)} = \{\sigma^{(i)}_{l\:m}(x)\}_{l,m=1,\dots n}$, $x\in\Omega$ satisfy \eqref{ellitticita'sigma}, and assume that $\sigma^{(i)}$ is constant on ${\overline{\Omega}}\cap B_{\rho}(y)$, for $i=1,2$.  If $\mathcal{N}^{\Sigma}_{\sigma^{(i)}}$ is the local N-D map corresponding to $\sigma^{(i)}$, for $i=1,2$, then
\begin{equation}\label{stima puntuale g}
\left|\left| \sigma^{(1)}(y) - \sigma^{(2)}(y)\right|\right|_{\mathcal{L}({\mathbb{R}}^n, {\mathbb{R}}^n)} \leq C E^{1-\beta} ||\mathcal{N}^{\Sigma}_{\sigma^{(1)}} - \mathcal{N}^{\Sigma}
_{\sigma^{(2)}}||^{\beta}_{\mathcal{L}\left(_{0}H^{-\frac{1}{2}}(\partial\Omega), _{0}H^{\frac{1}{2}}(\partial\Omega)\right)},
\end{equation}
where $C$ is a positive uniform constant and $\beta=\frac{1}{n-1}$.
\end{theorem}



\section{Construction of mollifiers on a graph and their $H^{-\frac{1}{2}}$-norm.}\label{mollifiers on graphs}

For any two points $\xi, x\in\Sigma$, with 
\begin{equation}\label{punti su Sigma}
\xi=(\xi',\varphi(\xi'))\quad ;\quad x=(x',\varphi(x'))
\end{equation}
and $\tau>0$, we denote
\begin{equation}\label{delta sul grafico}
\delta_{\tau}(\xi,x)=C_{\tau}(\xi',x')\delta\left(\frac{\xi' -x'}{\tau}\right)
\end{equation}
and we choose $C_{\tau}$ in such a way that
\begin{equation}\label{integrale delta 1}
\int_{\Sigma} \delta_{\tau}(\xi,x)dS(\xi)=1,\qquad\textnormal{whenever}\quad B_{\tau}(x')\subset B_{\rho}.
\end{equation}
To compute $C_{\tau}$ we form
\begin{equation}\label{determinazione Ctau 1}
1=\int_{\Sigma} \delta_{\tau}(\xi,x)dS(\xi)=\int_{B_{\rho}} C_{\tau}(\xi',x')\delta\left(\frac{\xi' - x'}{\tau}\right)\sqrt{1+|\nabla\varphi(\xi')|^2}d\xi'.
\end{equation}
Denoting $q(\xi')=\sqrt{1+|\nabla\varphi(\xi')|^2}$, we set
\begin{equation}\label{Ctau}
C_{\tau}(\xi',x')=\frac{\tau^{1-n}}{q(\xi')},
\end{equation}
so that
\begin{equation}\label{integrale delta 2}
\int_{\Sigma} \delta_{\tau}(\xi,x)dS(\xi)=\int_{\mathbb{\rho}^{n-1}}\tau^{1-n}\delta\left(\frac{\xi'-x'}{\tau}\right)d\xi' =1.
\end{equation}
\begin{remark}
Notice that $q(\xi')=q(x')+\mathcal{O}(\tau^{\alpha})$ on $B_{\tau}(x')$, hence $q(\xi') = \mathcal{O}(1)$ and therefore $C_{\tau}(\xi', x')=\mathcal{O}(\tau^{1-n})$.
\end{remark}

We define the $H^{-\frac{1}{2}}$-norm of an element  $f\in H^{-\frac{1}{2}}(\partial \Omega)$ as follows
\begin{equation}\label{norma H -1/2}
||f||^{2}_{H^{-\frac{1}{2}}(\partial\Omega)}=\int_{\partial \Omega} h(x)( f(x) -\bar{f} ) \:dS(x)= \int_{\Omega} |\nabla h(x)|^2\:dx,
\end{equation}
where 
\begin{equation}\label{f media integrale}
\bar{f}=\frac{1}{|\partial\Omega|} \int_{\partial\Omega}f(x)\:dS(x)
\end{equation} 
and $h$ solves
\begin{displaymath}
\left\{ \begin{array}{ll} \Delta h=0, &
\textrm{$\textnormal{in}\quad\Omega$},\\
\frac{\partial h}{\partial\nu}=f -\bar{f}, & \textrm{$\textnormal{on}\quad{\partial\Omega}.$}
\end{array} \right.
\end{displaymath}

Recall that, for $y\in\Omega$, the Neumann kernel $N^{\Omega}_{0}(\cdot,\:y)$ for the Laplacian $\Delta$ in $\Omega$ is defined, to be the distributional solution to
\begin{displaymath}\label{def neumann kernel}
\left\{ \begin{array}{ll}
\Delta\:N^{\Omega}_{0}(\cdot,y)=-\delta(\cdot -y), & \textnormal{in}\quad\Omega\\
\frac{\partial N^{\Omega}_{0}(\cdot, y)}{\partial\nu}= -\frac{1}{\vert\partial\Omega\vert},
& \textnormal{on}\quad{\partial\Omega},
\end{array} \right.
\end{displaymath}
where we impose the normalization
\[\int_{\partial\Omega} N^{\Omega}_{0}(\cdot,y)\:dS(\cdot)=0,\]
and that for $y\in\partial\Omega$, it solves

\begin{displaymath}
\left\{ \begin{array}{ll}
\Delta\:N^{\Omega}_{0}(\cdot,y)=0, & \textnormal{in}\quad\Omega\\
\frac{\partial N^{\Omega}_{0}(\cdot, y)}{\partial\nu}= \delta(\cdot -y) - \frac{1}{\vert\partial\Omega\vert},
& \textnormal{on}\quad{\partial\Omega}.
\end{array} \right.
\end{displaymath}

Hence we can write
\begin{eqnarray}
h(x) &=& \int_{\partial\Omega} N^{\Omega}_{0}(x,\xi) (f(\xi)-\bar{f})\:dS(\xi)=\int_{\partial\Omega} N^{\Omega}_{0}(x,\xi) f(\xi) \:dS(\xi),
\end{eqnarray}
since $N^{\Omega}_{0}(x, \cdot)$ and $h$ have zero average on $\partial\Omega$. Hence
\begin{equation}\label{norma H -1/2 bis}
||f||^{2}_{H^{-\frac{1}{2}}(\partial\Omega)}=\int_{\partial\Omega\times\partial\Omega} N^{\Omega}_{0}(x,\xi) f(\xi) f(x)\:dS(x)\:dS(\xi).
\end{equation}
We are now in the position to estimate the behaviour of the $H^{-\frac{1}{2}}(\partial\Omega)$ norm of the mollified delta function $\delta_{\tau}(\cdot,x)$, with $x\in\Sigma$, in terms of $\tau$.

\begin{lemma}
Given $x\in\Sigma$ such that $B_{\tau}(x')\subset B'_{\rho}$, we have
\begin{equation}\label{stima H -1/2 delta}
||\delta_{\tau}(\cdot, x)||^{2}_{H^{-\frac{1}{2}}(\partial\Omega)}\leq C \tau^{2-n},
\end{equation}
where $C>0$ is a uniform constant.
\end{lemma}

\begin{proof}
By \eqref{norma H -1/2 bis} and the fact that $\delta_{\tau}$ is compactly supported on $\Sigma$, we have
\begin{eqnarray}\label{estimate delta}
||\delta_{\tau}(\cdot, x)||^{2}_{H^{-\frac{1}{2}}(\partial\Omega)} &=& \int_{\Sigma\times\Sigma} N^{\Omega}_{0}(\xi,\eta) \delta_{\tau}(\xi ,x) \delta_{\tau}(\eta ,x) \:dS(\xi)\:dS(\eta)\nonumber\\
& \leq & C \int_{B'_{\rho} \times B'_{\rho}}|\xi' - \eta'|^{2-n}\tau^{2(1-n)}\delta\left(\frac{\xi' - x'}{\tau}\right) \delta\left(\frac{\eta' - x'}{\tau}\right)\:d\xi' \:d\eta'.
\end{eqnarray}
The change of variables
\[\zeta'=\frac{\xi'- x'}{\tau}\qquad ; \qquad \theta'=\frac{\eta' - x'}{\tau},\]
together with the fact that $|\xi'-\eta'|=\tau|\zeta' - \theta'|$, leads to
\begin{eqnarray}\label{estimate delta final}
||\delta_{\tau}(\cdot, x)||^{2}_{H^{-\frac{1}{2}}(\partial\Omega)} & \leq & C \int_{B'_1 \times B'_1}\tau^{2-n}|\zeta' - \theta'|^{2-n}\tau^{2(1-n)}\delta\left(\zeta'\right) \delta\left(\theta'\right)\tau^{2(n-1)}\:d\zeta'\:d\theta'\nonumber\\
&=& C \int_{B'_1 \times B'_1}\tau^{2-n}|\zeta' - \theta'|^{2-n}\delta\left(\zeta'\right) \delta\left(\theta'\right)\:d\zeta'\:d\theta'\leq C\tau^{2-n}.
\end{eqnarray}

\end{proof}

We recall that for $\sigma(x)=\left\{\sigma_{ij}(x)\right\}_{i,j=1,\dots , n}$, $x\in\Omega$, symmetric, positive definite matrix-valued function satisfying \eqref{ellitticita'sigma}, we denote by $L$ the operator
\begin{equation}\label{operator L}
L=\mbox{div}\left(\sigma\nabla\cdot\right)
\end{equation}
nd that if in dimension $n>2$ we define the matrix
\begin{equation}\label{g}
g = \left(\det\sigma\right)^{\frac{1}{n-2}}\sigma^{-1},
\end{equation}
then 
\[ \frac{1}{\sqrt {det g}} L = \Delta_g, \]
on the open set $\Omega$ endowed with the Riemannian metric $g$,  see for instance \cite{B-G-M}, \cite{U}.  We emphasize that, being $n>2$, the knowledge of $\sigma$ is equivalent to the knowledge of $g$.

\section{Stability of the tangential part of $g$}\label{stability tangential g}
We start by observing that \eqref{g}, together with the uniform ellipticity assumption \eqref{ellitticita'sigma} on $\sigma$, implies the following uniform ellipticity of $g$
\begin{eqnarray}\label{ellitticita' g}
\lambda^{-\frac{2n-2}{n-2}}|\xi|^2\leq g(x)\xi\cdot\xi\leq \lambda^{\frac{2n-2}{n-2}},
& &for\:almost\:every\:x\in\Omega,\nonumber\\
& &for\:every\:\xi\in\mathbb{R}^{n},
\end{eqnarray}
where $\lambda>0$ has been introduced in \eqref{ellitticita'sigma}.

We also recall below few facts from \cite{Al-dH-G} about the Neumann kernel to make this manuscript self-contained. The Neumann kernel $N^{\Omega}_{\sigma}$ for the boundary value problem associated with the operator \eqref{operator L} and $\Omega$, for any $y\in\Omega$, $N^{\Omega}_{\sigma}(\cdot,y)$, is defined to be the distributional solution to
\begin{displaymath}\label{def neumann kernel}
\left\{ \begin{array}{ll}
L\:N^{\Omega}_{\sigma}(\cdot,y)=-\delta(\cdot -y), & \textnormal{in}\quad\Omega\\
\sigma\nabla N^{\Omega}_{\sigma}(\cdot, y)\cdot\nu= -\frac{1}{\vert\partial\Omega\vert},
& \textnormal{on}\quad{\partial\Omega}.
\end{array} \right.
\end{displaymath}
Note that $N^{\Omega}_{\sigma}$ is uniquely determined up to an additive constant. For simplicity we impose the normalization
\[\int_{\partial\Omega} N^{\Omega}_{\sigma}(\cdot,y)\:dS(\cdot)=0.\]
With this convention we obtain by Green's identities that
\begin{equation}\label{symmetry of N}
N^{\Omega}_{\sigma}(x,y) = N^{\Omega}_{\sigma}(y,x),\qquad\textnormal{for\:all}\quad x,y\in\Omega,\quad x\neq y.
\end{equation}
$N^{\Omega}_{\sigma}(x,y)$ extends continuously up to the boundary $\partial\Omega$ (provided that $x\neq y$) and in particular, when $y\in\partial\Omega$, it solves
\begin{displaymath}
\left\{ \begin{array}{ll}
L\:N^{\Omega}_{\sigma}(\cdot,y)=0, & \textnormal{in}\quad\Omega\\
\sigma\nabla N^{\Omega}_{\sigma}(\cdot, y)\cdot\nu= \delta(\cdot -y)-\frac{1}{\vert\partial\Omega\vert},
& \textnormal{on}\quad{\partial\Omega}.
\end{array} \right.
\end{displaymath}

\begin{theorem}\label{theorem neumann function holder}
Let $y$, $\rho$ and $\Sigma$ be defined by \eqref{definizione Sigma}. If $L$ is the operator \eqref{operator L}, with coefficients matrix $\sigma\in C^{\alpha}(B_{\rho}(y)\cap\overline\Omega)$, with $0<\alpha<1$, then the Neumann kernel $N^{\Omega}_{\sigma}(\cdot,y)$ satisfies
\begin{equation}\label{neumann kernel holder coefficients}
N^{\Omega}_{\sigma}(x,y)= 2\Gamma_{\sigma}(x,y)+\mathcal{O}(|x-y|^{2-n+\alpha}),
\end{equation}
as $x\rightarrow y$, $x\in\overline{\Omega}\setminus{\{y\}}$. Here
\begin{equation}\label{def Gamma}
\Gamma_{\sigma}(x,y):=C_{n}\:\Big(g(y)(x-y)\cdot(x-y)\Big)^{\frac{2-n}{2}}
\end{equation}
and $C_{n}=\frac{1}{n(n-2)\omega_n}$ with $\omega_n$ denoting the volume of the unit ball in $\mathbb{R}^n$.
\end{theorem}
\begin{proof}
See \cite{Al-dH-G} for a proof. 
\end{proof}

Therefore, we have

\begin{lemma}\label{lemma tau tangential}
Let $y$, $\rho$ and $\Sigma$ be defined by \eqref{definizione Sigma}. If $L$ is the operator \eqref{operator L}, with coefficients matrix $\sigma\in C^{\alpha}(B_{\rho}(y)\cap\overline\Omega)$, with $0<\alpha<1$, then the knowledge of $N^{\Omega}_{\sigma}(x, y)$, for every $x\in\partial\Omega\cap B_{\rho}(y)$ uniquely determines
\begin{equation}\label{tau boundary}
g_{(n-1)}(y)=\left\{g(y)v_i \cdot v_j\right\}_{i,j=1,\dots , (n-1)},
\end{equation}
where $v_1,\dots ,v_{n-1}$  is a basis for $T_{y} (\partial\Omega)$, the tangent space to $\partial\Omega$ at $y$.
\end{lemma}








Such a uniqueness result obtained in \cite{Al-dH-G}, will guide us towards a stability estimate. 

In what follows, for $y\in\Sigma$, we set for $i=1,2$:
\begin{equation}\label{def Gamma i}
\Gamma_{i}(x,y):=\Gamma_{\sigma^{(i)}(y)}(x,y)=C_{n}\:\Big(g^{(i)}(y)(x-y)\cdot(x-y)\Big)^{\frac{2-n}{2}},
\end{equation}
\begin{equation}\label{def N i}
N_{i}(x,y):=N^{\Omega}_{{\sigma}_i}(x,y)
\end{equation}
\begin{equation}\label{operator L i}
L_i=\mbox{div}\left(\sigma^{(i)}\nabla\cdot\right)
\end{equation}
and
\begin{equation}\label{operator L i congelato 1}
L_{i;y}=\mbox{div}\left(\sigma^{(i)}(y)\nabla\cdot\right).
\end{equation}

\begin{lemma}\label{lemma stima gradiente N,K}
Under the same hypotheses of Theorem \ref{theorem neumann function holder}, for any $y\in\Sigma$ and any $x\in\overline{\Omega}\setminus \{y\}$  we have
\begin{eqnarray}
\left|\left(N_1 - N_2\right)(x,y)\right| &\leq & C E  |x-y|^{2-n},\label{stima N,K}\\
\left|\nabla_{x}(N_1 - N_2)(x,y)\right| &\leq & C E \left| x-y\right|^{1-n},\label{stima gradiente N,K}
\end{eqnarray}
where $C>0$ is a uniform constant.
\end{lemma}

\begin{proof}
From Theorem \ref{theorem neumann function holder} for any $y\in\Sigma$, for any $x\in\Omega$, for $i=1,2$, we have
\begin{equation}\label{N i}
N_i(x,y)=2\Gamma_i(x,y)+R_i(x,y),
\end{equation}
with
\begin{eqnarray}
|R_i(x,y)|\leq C |x-y|^{2-n+\alpha},\\\label{stima R}
|\nabla_x R_i(x,y)|\leq C |x-y|^{1-n+\alpha}.\label{stima grad R}
\end{eqnarray}
Recalling that for $y\in\Sigma$, $N_i(\cdot,y)$ is the distributional solution to 
\begin{displaymath}\label{neumann kernel boundary bis}
\left\{ \begin{array}{ll}
L_i\:N_{i}(\cdot,y)=0, & \textnormal{in}\quad\Omega\\
\sigma\nabla N_i(\cdot, y)\cdot\nu= \delta(\cdot -y)-\frac{1}{\vert\partial\Omega\vert},
& \textnormal{on}\quad{\partial\Omega}.
\end{array} \right.
\end{displaymath}
for any $\eta\in H^{1}(\Omega)\cap C^{1}(\overline{\Omega}\cap{B_{\rho}(y)})$, we have
\begin{eqnarray}\label{int 1}
\int_{\Omega} \sigma^{(i)}(x)\nabla_{x} N_i (x,y)\cdot\nabla_x\eta(x)dx & = & \int_{\partial\Omega} \eta(x)\left(\delta(x-y)-\frac{1}{|\partial\Omega|}\right)dS(x)\nonumber\\
& = & \eta(y)-\bar{\eta},
\end{eqnarray}
where $\bar{\eta}=\frac{1}{|\partial\Omega|}\int_{\partial\Omega} \eta(x) dS(x)$. Recalling the decomposition \eqref{N i} we obtain
\begin{equation}\label{int 2}
\int_{\Omega} \sigma^{(i)}(x)\nabla_{x} 2\Gamma_i (x,y)\cdot\nabla_x\eta(x)dx + \int_{\Omega} \sigma^{(i)}(x)\nabla_{x} R_i (x,y)\cdot\nabla_x\eta(x)dx = \eta(y)-\bar{\eta},
\end{equation}
which leads to
\begin{eqnarray}\label{int 3}
\int_{\Omega} \sigma^{(i)}(y)\nabla_{x} 2\Gamma_i (x,y)\cdot\nabla_x\eta(x)dx  &+&  \int_{\Omega}\left( \sigma^{(i)}(x) - \sigma^{(i)}(y)\right)\nabla_{x} 2\Gamma_i (x,y)\cdot\nabla_x\eta(x)dx\nonumber\\
&+&  \int_{\Omega}\sigma^{(i)}(x)\nabla_{x} R_i (x,y)\cdot\nabla_x\eta(x)dx= \eta(y)-\bar{\eta}.
\end{eqnarray}
Noticing that $2\Gamma_i(\cdot , y)$ solves the boundary value problem
\begin{displaymath}\label{neumann kernel boundary}
\left\{ \begin{array}{ll}
L_{i;y}\:2\Gamma_{i}(\cdot , y)=0, & \textnormal{in}\quad\Omega\\
\sigma^{(i)}(y) \nabla \left(2\Gamma_i(\cdot ,y)\right)\cdot\nu= \delta(\cdot -y)+f_i(\cdot,y),
& \textnormal{on}\quad{\partial\Omega}.
\end{array} \right.
\end{displaymath}
with
\begin{equation}\label{stima f i}
|f_i(x,y)|\leq C |x-y|^{1-n+\alpha},
\end{equation}
where $C>0$ is a uniform constant and  $L_{i;y}$ has been defined in \eqref{operator L i congelato 1}, therefore
\begin{equation}\label{int 4}
\int_{\Omega} \sigma^{(i)}(y)\nabla_{x} 2\Gamma_i (x,y)\cdot\nabla_x\eta(x)dx = \eta(y) + \int_{\partial\Omega} f_i (x,y)\eta(x) dS(x).
\end{equation}
and defining
\begin{equation}
F_i(x,y):= \left( \sigma^{(i)}(x)- \sigma^{(i)}(y)\right)\nabla_{x} 2\Gamma_i (x,y),
\end{equation}
we can rewrite \eqref{int 3} as
\begin{eqnarray}\label{int 5}
 \int_{\Omega} \sigma^{(i)}(x)\nabla_{x} R_i (x,y)\cdot\nabla_x\eta(x)dx  & = & -  \int_{\Omega}F_i (x,y)\cdot\nabla_x\eta(x)dx\nonumber\\ & - & \int_{\partial\Omega} f_i (x,y)\eta(x) dS(x) -\bar{\eta},
\end{eqnarray}
where $F_i(x,y)$ is zero for $x\in B_{\rho}(y)$, therefore bounded for $x\in\Omega$. Therefore by combining \eqref{int 4} and \eqref{int 5}, we have that  for any $\eta\in H^1(\Omega)\cap C^1(\overline{B_{\rho}(y)})$ 
\begin{eqnarray}\label{int 6}
& & \int_{\Omega} \sigma^{(1)}(x)\nabla_{x} \left(R_1 - R_2\right)(x,y)\cdot\nabla_x\eta(x)dx \nonumber\\
&  & = \int_{\Omega} \left(\sigma^{(2)} - \sigma^{(1)}\right)(x) \nabla_x R_2(x,y)\cdot\nabla_x\eta(x)dx\nonumber\\
& & -  \int_{\Omega}\left(F_1 - F_2\right)(x,y)\cdot\nabla_x\eta(x)dx  -  \int_{\partial\Omega}\left(f_1 - f_2\right) (x,y)\eta(x) dS(x).
\end{eqnarray}
By denoting for any $x\in\Omega$, for any $y\in\Sigma$,
\[F(x,y):= \left(F_1 - F_2\right)(x,y)\quad ; \quad f(x,y):=\left(f_1 - f_2\right)(x,y),\]
we have
\begin{eqnarray}
& & |F(x,y)|\leq C E,\label{stima F}\\
& & |f(x,y)|\leq C E |x-y|^{1-n+\alpha},\label{stima f}
\end{eqnarray}
and by denoting also
\[G(x,y):=  \left(\sigma^{(2)} - \sigma^{(1)}\right)(x) \nabla_x R_2(x,y) - F(x,y), \]
we have
\begin{equation}\label{stima G}
|G(x,y)|\leq C E |x-y|^{1-n+\alpha},
\end{equation}
where $C>0$ denotes a uniform constant.\\
For any $z\in\overline{\Omega}\setminus \{y\}$ and any $k\in \mathbb{N}$, we define $\eta_k(x)=\min\{N_1(x,z), k\}$ with $x\in\Omega$.  By choosing $\eta(x)=\eta_k(x)$ in  \eqref{int 6} and  by the dominated convergence theorem we obtain
\begin{eqnarray}\label{int 6 bis}
 \int_{\Omega} \sigma^{(1)}(x)\nabla_{x} \left(R_1 - R_2\right)(x,y)\cdot\nabla_x N_1(x,z)dx
= \int_{\Omega} G(x,y)\cdot \nabla_x N_1(x,z)dx  -  \int_{\partial\Omega} f(x,y)N_1(x,z) dS(x).
\end{eqnarray}

By performing integration by parts on the integral appearing on the left hand side of \eqref{int 6} we have
\begin{eqnarray}\label{int 7}
&  &  \frac{1}{|\partial\Omega|}\int_{\partial\Omega} \left(R_1 - R_2\right)(x,y)dS(x) + \left(R_1 - R_2\right)(z,y)\nonumber\\
& & = \int_{\Omega} G(x,y) \cdot \nabla_xN_1(x,z)dx  -  \int_{\partial\Omega} f(x,y)N_1(x,z) dS(x),
\end{eqnarray}
where  the right hand side of equality \eqref{int 7} can be estimated as follows
\begin{eqnarray}\label{int 8}
& &\left| \int_{\Omega} G(x,y) \cdot \nabla_xN_1(x,z)dx  -  \int_{\partial\Omega} f(x,y)N_1(x,z) dS(x)\right|\nonumber\\
& & C ||\sigma^{(1)}-\sigma^{(2)}||_{L^{\infty}(\Omega)}\bigg\{\int_{\Omega} |x-y|^{1-n+\alpha}\:|x-z|^{1-n} dx\nonumber\\
& & + \int_{\partial\Omega} |x-y|^{1-n+\alpha}\:|x-z|^{2-n} dS(x)\bigg\}\nonumber\\
& & \leq C E |z-y|^{2-n+\alpha},
\end{eqnarray}
where $C>0$ is a uniform constant. To estimate $\int_{\partial\Omega}\left(R_1 - R_2\right)(x,y)dS(x)$ in \eqref{int 6} recall that $N_i(\cdot, y)$ is uniquely determined by imposing the condition
\begin{equation}\label{N i unico}
\int_{\partial\Omega} N_i (x,y) dS(x) = 0,
\end{equation}
which leads to
\begin{equation}\label{R i}
\int_{\partial\Omega} R_i (x,y) dS(x) = -\int_{\partial\Omega} 2\Gamma_i (x,y) dS(x).
\end{equation}
Therefore
\begin{eqnarray}\label{int 9}
\left| \int_{\partial\Omega} \left(R_1 - R_2\right)(x,y)dS(x)\right| & = & \left|\int_{\partial\Omega} 2\left(\Gamma_1 - \Gamma_2\right)(x,y)dS(x)\right| \le C  E,
\end{eqnarray}
where $C>0$ is a uniform constant, which leads to
\begin{eqnarray}\label{estimate R1-R2}
\left|\left(R_1 - R_2\right)(z,y)\right| & \leq & C E \left\{1+  |z-y|^{2-n+\alpha}\right\} \leq  C E |z-y|^{2-n+\alpha},
\end{eqnarray}
where $C>0$ is a uniform constant. \eqref{stima N,K} follows from \eqref{estimate R1-R2}. To prove \eqref{stima gradiente N,K} we observe that $\sigma^{(i)}$ is constant on $\overline{\Omega}\cap B_{\rho}(y)$, for $i=1,2$, therefore
\begin{equation}\label{div traccia}
\mbox{div}\left(\sigma^{(i)}\nabla N_i(\cdot,y)\right) = \mbox{tr}\left(\sigma^{(i)}D^2 N_i (\cdot, y)\right)=0,\qquad\textnormal{for}\quad i=1,2,
\end{equation}
therefore
\begin{equation}\label{traccia traccia}
\mbox{tr}\left(\sigma^{(1)}D^2 \left(N_1 - N_2\right) (\cdot, y)\right) = \mbox{tr}\left((\sigma^{(2)} - \sigma^{(1)})D^2 N_2 (\cdot, y)\right).
\end{equation}
By fixing $r>0$, with $r < \frac{\rho}{4}$ and defining 
\begin{eqnarray*}
& &C_r^{+}:=\left(B_{2r}(y)\setminus\overline{B_r(y)}\right)\cap\Omega\\
& &C_{2r}^{+}:=\left(B_{4r}(y)\setminus\overline{B_{\frac{r}{2}}(y)}\right)\cap\Omega,
\end{eqnarray*}
we have for $i=1,2$
\begin{equation}\label{stima D2 N}
||D^2 N_i||_{C^{\alpha}(C_r^{+})}\leq \frac{C}{r^2} ||N_i||_{C^{\alpha}(C_{2r}^{+})}\leq C r^{-n-\alpha}
\end{equation}
and
\begin{equation}\label{stima D2 N1-N2}
||D^2 (N_1 - N_2)||_{C^{\alpha}(C_r^{+})}\leq  C  E r^{-n-\alpha}.
\end{equation}
\eqref{stima gradiente N,K} follows by \eqref{stima N,K}, \eqref{stima D2 N1-N2} and an interpolation argument.

\end{proof}




\begin{proposition}[Stability of the tangential part of $g$]\label{lemma neumann map and K}

For any $y\in\Sigma$, 
\begin{equation}\label{stima puntuale g tangenziale}
\left|\left| g^{(1; n-1)}(y) - g^{(2; n-1)}(y)\right|\right|_{\mathcal{L}(\mathbb{R}^{n-1},\  \mathbb{R}^{n-1} )} \leq C E^{1-\beta} ||\mathcal{N}^{\Sigma}_{\sigma_1} - \mathcal{N}^{\Sigma}
_{\sigma_2}||^{\beta}_{\mathcal{L}\left(_{0}H^{-\frac{1}{2}}(\partial\Omega), _{0}H^{\frac{1}{2}}(\partial\Omega)\right)},
\end{equation}
where $C$ is a positive uniform constant, $\beta=\frac{1}{n-1}$, $g^{(i; n-1)}(y)$ is the $(n-1)\times (n-1)$ upper left sub-matrix of $g^{(i)}(y)$, for $i=1,2$.
\end{proposition}

\begin{proof}
Let $d>0$ be such that $d<(1+M)\rho$.
Given distinct points $x,y,w,z\in\Sigma$, we recall from \cite{Al-dH-G} the following definition
\begin{equation}\label{def K}
K_{\sigma}(x,y,w,z)=N_{\sigma}(x,y)-N_{\sigma}(x,w)-N_{\sigma}(z,y)+N_{\sigma}(z,w),
\end{equation}
We also recall that knowing $\mathcal{N}^{\Sigma}_{\sigma}$ is equivalent to knowing $K_{\sigma}$, for any $x,y,w,z\in\Sigma$ \cite[Lemma 3.8]{Al-dH-G}.

We also note that, fixing $w,z\in\Sigma$, $K_{\sigma}$, as a function of $x$, $y$, has the same asymptotic behaviour of $N_{\sigma}(x,y)$ as $x\rightarrow y$.

Given $y\in\Sigma$, we choose $x,w,z\in\Sigma$ such that
\begin{eqnarray}
|x-y| &\leq & \frac{d}{4}\label{distanza x y}\\
|x-w| &\geq &\frac{d}{4}\label{distanza x w}\\
|x-z| &\geq & \frac{d}{4}\label{distanza x z}\\
|w-z| &\geq & \frac{d}{4}.\label{distanza w z}
\end{eqnarray}
Let 
\begin{equation}\label{tau}
\tau := h|x-y|,\qquad\textnormal{with}\quad 0<h<\frac{1}{16}
\end{equation}
and let $\delta_{\tau}(\cdot ,\cdot)$ be the approximate Dirac's delta functions on $\Sigma$ introduced in section \ref{mollifiers on graphs}, centered on the second argument. Then we have
\begin{eqnarray}\label{uguaglianza mappa nucleo}
&&\quad\quad  \langle\delta_{\tau}(\cdot\:,x) - \delta_{\tau}(\cdot\:,z), \left(\mathcal{N}^{\Sigma}_{\sigma_1} - \mathcal{N}^{\Sigma}_{\sigma_2}\right)\left(\delta_{\tau}(\cdot\:,y) - \delta_{\tau}(\cdot\:,w)\right)\rangle \\
&&=\int_{\Sigma\times \Sigma}\left(N_1 - N_2\right) (\xi,\eta)\left(\delta_{\tau}(\xi\:,x) - \delta_{\tau}(\xi\:,z)\right)\left(\delta_{\tau}(\eta\:,y) - \delta_{\tau}(\eta\:,w)\right)dS(\xi)dS(\eta)\nonumber 
\end{eqnarray}
The integral appearing on the right hand side of \eqref{uguaglianza mappa nucleo} depends on $x,y,w,z$ and to estimate how close this quantity is to $(K_1-K_2)(x,y,w,z)$ we form
\begin{eqnarray}\label{int K 1}
&&\quad \quad \left(K_1-K_2\right)(x,y,w,z) \\
&&- \int_{\Sigma\times \Sigma}\left(N_1 - N_2\right) (\xi,\eta)\left(\delta_{\tau}(\xi,x) - \delta_{\tau}(\xi,z)\right)\left(\delta_{\tau}(\eta,y) - \delta_{\tau}(\eta,w)\right)dS(\xi)dS(\eta)\nonumber\\
&&=\left(N_1-N_2\right)(x,y) - \int_{\Sigma\times \Sigma}\left(N_1 - N_2\right) (\xi,\eta)\delta_{\tau}(\xi,x)\delta_{\tau}(\eta,y)dS(\xi)dS(\eta)\nonumber\\
&&- \left(N_1-N_2\right)(x,w)- \int_{\Sigma\times \Sigma}\left(N_1 - N_2\right) (\xi,\eta)\delta_{\tau}(\xi,x)\delta_{\tau}(\eta,w)dS(\xi)dS(\eta)\nonumber\\
&&+\left(N_1-N_2\right)(z,w) - \int_{\Sigma\times \Sigma}\left(N_1 - N_2\right) (\xi,\eta)\delta_{\tau}(\xi,z)\delta_{\tau}(\eta,w)dS(\xi)dS(\eta)\nonumber\\
&&-\left(N_1-N_2\right)(z,y)- \int_{\Sigma\times \Sigma}\left(N_1 - N_2\right) (\xi,\eta)\delta_{\tau}(\xi,z)\delta_{\tau}(\eta,y)dS(\xi)dS(\eta).\nonumber
\end{eqnarray}
We estimate each term on the right hand side of \eqref{int K 1} as follows
\begin{eqnarray}\label{term 1}
&&\quad \quad \left|\left(N_1-N_2\right)(x,y) - \int_{\Sigma\times \Sigma}\left(N_1 - N_2\right) (\xi,\eta)\delta_{\tau}(\xi,x)\delta_{\tau}(\eta,y)dS(\xi)dS(\eta)\right|\\
&&=\left|\int_{\Sigma\times \Sigma}\Big(\left(N_1-N_2\right)(x,y) - \left(N_1 - N_2\right)(\xi,\eta)\Big)\delta_{\tau}(\xi,x)\delta_{\tau}(\eta,y)dS(\xi)dS(\eta)\right| \nonumber\\
&&\leq\int_{\Sigma\times \Sigma}\Big|\left(N_1-N_2\right)(x,y) - \left(N_1-N_2\right)(\xi,y)\Big| \delta_{\tau}(\xi,x)\delta_{\tau}(\eta,y)dS(\xi)dS(\eta)\nonumber\\ 
&&+\int_{\Sigma\times \Sigma}\Big| \left(N_1 - N_2\right)(\xi,y) - \left(N_1 - N_2\right)(\xi,\eta)\Big|\delta_{\tau}(\xi,x)\delta_{\tau}(\eta,y)dS(\xi)dS(\eta)\nonumber\\
&&\leq\int_{\Sigma\times \Sigma}\Big|\nabla_{\xi}\left(N_1-N_2\right)(\hat\xi,y)\Big||x-\xi| \delta_{\tau}(\xi,x)\delta_{\tau}(\eta,y)dS(\xi)dS(\eta)\nonumber\\ 
&&+\int_{\Sigma\times \Sigma}\Big|\nabla_{\eta}\left(N_1-N_2\right)(\xi,\hat\eta)\Big||\eta-y| \delta_{\tau}(\xi,x)\delta_{\tau}(\eta,y)dS(\xi)dS(\eta)\nonumber\\
&&\leq\int_{\Sigma\times \Sigma}\frac{CE}{|\hat\xi -y|^{n-1}}|x-\xi| \delta_{\tau}(\xi,x)\delta_{\tau}(\eta,y)dS(\xi)dS(\eta)\nonumber\\ 
&&+\int_{\Sigma\times \Sigma}\frac{CE}{|\xi -\hat\eta|^{n-1}}|\eta-y| \delta_{\tau}(\xi,x)\delta_{\tau}(\eta,y)dS(\xi)dS(\eta)\nonumber
\end{eqnarray}

where $\hat{\xi}=(1-t)x +t\xi$ and $\hat{\eta}=(1-s)y+s \eta$ for some $t,s\in(0,1)$.

Notice that, given the choice of $\tau$ in \eqref{tau}, we have that
\begin{equation}\label{stima tau}
\tau<\frac{d}{64};\qquad  |x-y| = \mathcal{O}(|\hat\xi - y|);\qquad  |x-y|= \mathcal{O}(|\xi-\hat\eta|),
\end{equation}
therefore
\begin{eqnarray}\label{term 1 tris}
&&\quad \quad \left|\left(N_1-N_2\right)(x,y) - \int_{\Sigma\times \Sigma}\left(N_1 - N_2\right) (\xi,\eta)\delta_{\tau}(\xi,x)\delta_{\tau}(\eta,y)dS(\xi)dS(\eta)\right|  \\
&&\leq\int_{\Sigma\times \Sigma}\frac{2\tau CE}{|x -y|^{n-1}}\delta_{\tau}(\xi,x)\delta_{\tau}(\eta,y)dS(\xi)dS(\eta) \nonumber\\ 
&&\leq\int_{\Sigma\times \Sigma}\frac{CEh}{|x -y|^{n-2}}\delta_{\tau}(\xi,x)\delta_{\tau}(\eta,y)dS(\xi)dS(\eta) \nonumber\\ 
&&=\frac{CEh}{|x -y|^{n-2}}.\nonumber
\end{eqnarray}
The remaining three terms appearing on the right hand side of \eqref{int K 1} only involve at most one of the poles $x,y\in\Sigma$ at the time and are therefore bounded by $CEh$. Therefore we have
\begin{eqnarray}\label{int K 2}
\bigg|\left(K_1-K_2\right)(x,y,w,z) & &\\
- \int_{\Sigma\times \Sigma}\left(N_1 - N_2\right) (\xi,\eta)\left(\delta_{\tau}(\xi,x) - \delta_{\tau}(\xi,z)\right)\left(\delta_{\tau}(\eta,y) - \delta_{\tau}(\eta,w)\right)dS(\xi)dS(\eta)\bigg|& &\nonumber\\ 
\leq \frac{CEh}{|x-y|^{n-2}} + CEh & & \nonumber\\
\leq\frac{CEh + CEhd^{n-2}}{|x-y|^{n-2}}\leq\frac{CEh}{|x-y|^{n-2}}. & &\nonumber
\end{eqnarray}
Recalling that
\begin{equation}\label{stima delta}
\big|\langle\delta_{\tau}(\cdot\:,x) - \delta_{\tau}(\cdot\:,z), \left(\mathcal{N}^{\Sigma}_{\sigma_1} - \mathcal{N}^{\Sigma}_{\sigma_2}\right)\left(\delta_{\tau}(\cdot\:,y) - \delta_{\tau}(\cdot\:,w)\right)\rangle\big|\leq C\tau^{2-n}\varepsilon , 
\end{equation}
where 
$\varepsilon=||\mathcal{N}^{\Sigma}_{\sigma^{(1)}}-\mathcal{N}^{\Sigma}_{\sigma^{(2)}}||_{\mathcal{L}\left(_{0}H^{-\frac{1}{2}}(\partial\Omega), _{0}H^{\frac{1}{2}}(\partial\Omega)\right)}$.
Hence, recalling the definition of $\tau$ in \eqref{tau}, we obtain the following pointwise estimate for $K_1 - K_2$
\begin{eqnarray}\label{stima K1-K2}
|(K_1 - K_2)(x,y,w,z)| &\leq &C\varepsilon{\tau}^{2-n} +\frac{CEh}{|x-y|^{n-2}}\\
&=&C\varepsilon\frac{{h}^{2-n}}{{|x-y|}^{n-2}} +\frac{CEh}{|x-y|^{n-2}}\nonumber\\
& = &\frac{ C\varepsilon h^{2-n} +CEh }{|x-y|^{n-2}} ,\qquad 0<h<\frac{1}{16}\nonumber.
\end{eqnarray}
Minimization of \eqref{stima K1-K2} with respect to $h$ leads to
\begin{equation}\label{stima K1-K2 bis}
|(K_1 - K_2)(x,y,w,z)|\leq \frac{C\varepsilon^{\beta}E^{1-\beta}}{|x-y|^{n-2}},
\end{equation}
that is
\begin{equation}\label{stima K1-K2 tris}
|x-y|^{n-2}|(K_1 - K_2)(x,y,w,z)|\leq C\varepsilon^{\beta}E^{1-\beta},
\end{equation}
with $\beta=\frac{1}{n-1}$. Inequality \eqref{stima K1-K2 tris} is a uniform bound with respect to $x,y\in\Sigma$. {Setting $y=0$, $\nu(0)=-e_{n}$, where $\{e_1,\dots , e_n\}$ denotes the canonical basis of $\mathbb{R}^n$, and writing $x\in\Sigma$ as $x=(x',\varphi(x'))$, with $x' = r\xi'$, with $\xi'\in\mathbb{R}^{n-1}$ and $||\xi'|| =1$, \eqref{neumann kernel holder coefficients} leads to
\begin{equation}\label{stima intermedia g}
\left|\left(g^{(1)}(0)\xi'\cdot\xi'\right)^{\frac{2-n}{2}}  -  \left(g^{(2)}(0)\xi'\cdot\xi'\right)^{\frac{2-n}{2}}\right|\leq C_n \left|\lim_{r\rightarrow 0} r^{n-2}\left(K_1 - K_2\right)(x,0,z,w)\right|.
\end{equation}
where we identified $\xi'$ with $(\xi',0)$.
By Lagrange's theorem and \eqref{ellitticita' g}, we have
\begin{equation}\label{Lagrange}
\left|\left( g^{(1)}(0) - g^{(2)}(0)\right)\xi'\cdot\xi'\right|\leq C \left|\left(g^{(1)}(0)\xi'\cdot\xi'\right)^{\frac{2-n}{2}}  -  \left(g^{(2)}(0)\xi'\cdot\xi'\right)^{\frac{2-n}{2}}\right|
\end{equation}
and by combining \eqref{Lagrange} with \eqref{stima K1-K2 tris} together with \eqref{stima intermedia g}, we obtain
\begin{equation}\label{stima g tan final}
\left|\left( g^{(1)}(0) - g^{(2)}(0)\right)\xi'\cdot\xi'\right|\leq CE^{1-\beta}\varepsilon^{\beta},\quad\textnormal{for\:any}\quad\xi',\:\xi'\in\mathbb{R}^{n-1},\quad ||\xi'||=1,
\end{equation}
}
which concludes the proof,  since, as it is well known 
\begin{eqnarray}
\left|\left| g^{(1; n-1)}(0) - g^{(2; n-1)}(0)\right|\right|_{\mathcal{L}(\mathbb{R}^{n-1},\  \mathbb{R}^{n-1} )} = \sup\displaystyle{ \bigg \{ \left|\left( g^{(1)}(0) - g^{(2)}(0)\right)\xi'\cdot\xi'\right| \ : \xi'\in\mathbb{R}^{n-1}, \ \|\xi'\|=1 \bigg \}}.
\end{eqnarray}

\end{proof}



\section{Stability of the full metric $g$}\label{stability full g}

In the following we shall prove that up to a suitable condition on the geometry of $\Omega$ and on the structure of the metric $g$, the knowledge of the Neumann kernel in a neighborhood of $\Sigma$ allows us to recover the full metric $g$ on $\Sigma$.\\
We assume that $\Sigma$ is $C^{2,\alpha}$ and non-flat as per definition \ref{definition non-flatness}. This means that we can find $P_1, P_2, P_3\in\Sigma$ and a constant $C_0$, $0<C_0<1$ satisfying \eqref{nu1} - \eqref{nu3}. If we denote by $\{e_1,\dots , e_n\}$ the canonical basis in $\mathbb{R}^n$, we can assume, without loss of generality, that $P_1=0\in\Sigma$, that the tangent space to $\partial\Omega$ at $0\in\Sigma$ is $T_0 (\partial\Sigma)=\langle e_1,\dots , e_{n-1}\rangle$ and the outer unit normal to $\partial\Omega$ at $0$ is $\nu(P_1)= -e_n$.  Hence we have
\begin{equation}\label{c_01}
\nu(0)\cdot\nu(P_2) \le 1-C_0,
\end{equation}
\begin{equation}\label{c_02}
 \nu(0)\cdot\nu(P_3)\le 1-C_0,
 \end{equation}
\begin{equation}\label{c_03}
 \nu(P_2)\cdot\nu(P_3) \le 1-C_0 \ 
\end{equation}
and without loss of generality, we can assume that there exists some $0<\gamma_1<\frac{1}{2}$ such that 
\begin{equation}\label{normal P}
\nu(P_2)=\frac{1}{\sqrt{1+{\gamma_1}^2}}\left(-e_n+{\gamma_1}e_{n-1}\right).
\end{equation}
and  some $0\le\gamma_2,\gamma_3\le 1$ such that 
\begin{equation}\label{normal P tilde}
\nu(P_3)=\frac{1}{{\sqrt{1+{\gamma_2}^{2}+{\gamma_3}^2}}}\left(-e_n+\gamma_2 e_{n-1} + \gamma_3 e_{n-2}\right).
\end{equation}
 Let $\Pi_1, \Pi_2, \Pi_3$ be the tangent spaces to $\partial \Omega$ at { $0$, $P_2$, $P_3$}, respectively, with orthonormal basis $\{v_1^1,\dots,v_{n-1} ^1\},$  $ \{v_1^2,\dots,v_{n-1} ^2\},   $ $\{v_1^3,\dots,v_{n-1} ^3\}$
and a linear application $T$,
\begin{equation}
T : Sym_{n}\rightarrow \mathbb{R}^{{3}(n-1)^2},
\end{equation}\label{T}
defined by
\begin{equation}
Tg=\left\{\langle g v_i^k, v_j^k\rangle \: | \: k=1,2,3,\quad i,j=1,\dots, n-1\right\},\ \ \mbox{for any} \ g\in  Sym_{n}\ .
\end{equation}

\begin{proposition}\label{full stability metric}
In the above setting for $\Sigma$, if $\sigma\in L^{\infty}(\Omega\:,Sym_{n})$ satisfies \eqref{ellitticita'sigma} and it is constant on $\overline{\Omega}\cap B_{\rho}$, then the knowledge of $N^{\Omega}_{\sigma}(x,y)$, for every $x,y\in\Sigma$ uniquely determines $g(0)$. Moreover $T$ is a linear and injective application such that 
\begin{eqnarray}\label{norm}
{\|g\|_{{\mathcal{L}({\mathbb{R}}^n, {\mathbb{R}}^n)}}\le C \|Tg\|,}
\end{eqnarray}
where $C>0$ is a constant only depending on $C_0$ and $\|Tg\|$ is the Euclidean norm of $Tg$ in $\mathbb{R}^{3(n-1)^2}$.
\end{proposition}

\begin{proof}
We reformulate conditions \eqref{c_01} - \eqref{c_03} in terms of $\gamma_1$, $\gamma_2$, $\gamma_3$. From condition \eqref{c_01} we get
\begin{equation}
\nu(0)\cdot\nu(P_2) = \frac{1}{\sqrt{1+ \gamma_1^2}}\le 1-C_0,
\end{equation}
which leads to 
\begin{equation}\label{condizione A}
\gamma_1\ge k_0
\end{equation}
where $k_0=\sqrt{\frac{1-(1-C_0)^2}{(1-C_0)^2}}$. Furthermore, from \eqref{c_02} we obtain
\begin{equation}
\nu(0)\cdot\nu(P_3) = \frac{1}{{\sqrt{1+\gamma_2^2+ \gamma_3^2 }}} \le 1 - C_0,
\end{equation}
which leads to  
\begin{equation}\label{1}
\gamma_2^2 + \gamma_3^2 \ge {k_0}^2 \  .
\end{equation}
Hence we get that 
\begin{eqnarray}\label{condizione B}
\gamma_2\ge \frac{k_0}{\sqrt{2}}\ \ \
\mbox{or}
\ \ \gamma_3\ge \frac{k_0}{\sqrt{2}} \ . 
\end{eqnarray}
Finally, by condition \eqref{c_03}, we have  
\begin{equation}\label{thirdcondition}
\nu(P_2)\cdot\nu(P_3) = \frac{1+\gamma_1 \gamma_2}{\sqrt{1+ \gamma_1^2}\sqrt{1+\gamma_2^2}}\frac{\sqrt{1+{\gamma_2}}^2}{\sqrt{1+{\gamma_2}^2+ \gamma_3^2}}\le 1- C_0 \ .
\end{equation}  
By \eqref{thirdcondition} we have that 
\begin{eqnarray}\label{2}
& & (1+\gamma_1\gamma_2)^2\nonumber\\
& \le &(1-C_0)^2(1+\gamma_1^2)(1+\gamma_2^2 +\gamma_3^2) \nonumber\\
& \le &(1-C_0)^2(\gamma_1-\gamma_2)^2 + (1-C_0)^2(1+2\gamma_1\gamma_2 +\gamma_1^2\gamma_2^2 +\gamma_3^2 +\gamma_1^2\gamma_3^2).
\end{eqnarray}
Hence 
\begin{eqnarray}\label{3}
&&[1- (1-C_0)^2](1+\gamma_1\gamma_2)^2 - (1-C_0)^2\gamma_3^2(1+\gamma_1^2) \le (1-C_0)^2(\gamma_1-\gamma_2)^2.
\end{eqnarray}
If $\gamma_3\le \frac{k_0}{\sqrt{2}}$, from \eqref{3} and the fact that $\gamma_1\le \frac{1}{2}$ we get that 
\begin{eqnarray}\label{4}
&&[1- (1-C_0)^2](1+\gamma_1\gamma_2)^2 - \frac{5}{8}(1-C_0)^2{k_0}^2 \le (1-C_0)^2(\gamma_1-\gamma_2)^2, \ 
\end{eqnarray}
which leads to 
\begin{eqnarray}
(\gamma_1-\gamma_2)^2\ge \frac{3}{8}\frac{[1- (1-C_0)^2]}{2(1-C_0)^2} \ .
\end{eqnarray}
In conclusion, we have the following two cases: 
\begin{eqnarray}\label{condizionea}
\gamma_1\ge k_0 \ , \quad \gamma_3\ge \frac{k_0}{\sqrt{2}}
\end{eqnarray}
or 
\begin{eqnarray}\label{condizioneb}
\gamma_1\ge k_0 \ ,\quad \gamma_2\ge \frac{k_0}{\sqrt{2}}\ \ \ \mbox{and}\ \ |\gamma_1 -\gamma_2|\ge \sqrt{\frac{3[1- (1-C_0)^2]}{16(1-C_0)^2}} . 
\end{eqnarray}
We observe that the spaces $\Pi_1,\Pi_2,\Pi_3$ are generated by 
\begin{eqnarray*}
& & \left\{ e_1,\dots,e_{n-1}\right\},\\
& & \left\{ e_1,\dots,e_{n-2},\frac{e_{n-1}+\gamma_1 e_n}{\sqrt{1+\gamma_1^2}}\right\},\\
& &  \left\{ e_1,\dots,e_{n-3}, e_{n-2}+\gamma_3 e_n, e_{n-1}+\gamma_2 e_n\right\},
\end{eqnarray*}
respectively, although the latter is not an orthonormal basis. Let $\{v_1, \dots, v_{n-1}\}$ be an orthonormal basis for $\Pi_3$ and let 
\[w_1=e_{n-2}+\gamma_3 e_n, \qquad w_2=e_{n-1}+\gamma_2 e_n\ .\] 
By using the repeated index notation we have 
 \begin{equation}\label{w}
w_i = \langle w_i, v_j \rangle v_j,\qquad\textnormal{for}\quad i=1,2  
\end{equation} 
and 
\begin{eqnarray}\label{grepeatedindex} 
\langle gw_i,w_j\rangle = \langle g v_l,v_k\rangle \langle w_i,v_l\rangle \langle w_j,v_k\rangle ,  
\end{eqnarray}
for $i,j=1,2$ and  $k,\: l=1,\dots, n-1$.
Noticing that 
\begin{eqnarray}
\sum_{j=1}^{n-1}\langle w_i,v_j\rangle^2 = \|w_i\|^2\le \max\{1+\gamma_2^2,1+\gamma_3^2 \}\le 2 ,
\end{eqnarray}
by \eqref{grepeatedindex} we have that 
\begin{eqnarray}
|\langle gw_i,w_j\rangle |\le C \|Tg\|,
\end{eqnarray}
where $C$ is an absolute constant. Now, from the tangential component of $g$ over $\Pi_1$, we recover the upper left submatrix of $g$, in particular 
\begin{eqnarray}\label{stimaminore}
|g_{i,j}|=|\langle ge_i,e_j \rangle |\le \|Tg\|,
\end{eqnarray}
for any $i,j=1,\dots, n-1$. From the tangential component of $g$ over $\Pi_2$ we know the following quantities 
\begin{eqnarray}
\bigg\langle g e_i, \frac{e_{n-1}+\gamma_1 e_n }{\sqrt{1+\gamma_1^2}}\bigg\rangle =\frac{1} {\sqrt{1+\gamma_1^2}}(g_{n-1,i} + \gamma_1 g_{n,i}),
\end{eqnarray}
with $i=1,\dots, n-2$ . 
Hence since by \eqref{stimaminore} we can control $g_{n-1,i} $, we get that 
\begin{eqnarray}\label{stimaultimacolonnan-2}
|g_{i,n}|\le C \|Tg\|,
\end{eqnarray}
for any $i=1,\dots, n-2$. To estimate the remaining entries $g_{n-1 , n}$, $g_{nn}$ of $g\in Sym_n$, we consider the following known quantity
\begin{equation}
\bigg\langle g\left( \frac{e_{n-1}+\gamma_1 e_n }{\sqrt{1+\gamma_1^2}}\right), \left(\frac{e_{n-1}+\gamma_1 e_n }{\sqrt{1+\gamma_1^2}}\right)\bigg\rangle=
\frac{1}{1+\gamma_1^2}( g_{n-1,n-1} +2\gamma_1 g_{n-1,n} + \gamma_1^2 g_{n,n} ).
\end{equation}
Hence, since by \eqref{stimaminore} we can control $g_{n-1,n-1}$, we get 
\begin{eqnarray}\label{stimasomma}
|2g_{n-1,n} + \gamma_1 g_{n,n}|=|F_1|\le C \|Tg\|,\qquad\textnormal{for}\quad i=1,\dots, n-2,
\end{eqnarray}
where 
\[F_1= \frac{1}{\gamma_1}\Big( \big\langle  g ({e_{n-1}+\gamma_1 e_n }), {e_{n-1}+\gamma_1 e_n }\big\rangle  -  g_{n-1,n-1} \Big) .\]
From the tangential component of $g$ over $\Pi_3$ we know the quantity 
\begin{equation}
\bigg\langle g \left(\frac{e_{n-2}+\gamma_3 e_n }{\sqrt{1+\gamma_3^2}}\right), \left(\frac{e_{n-2}+\gamma_3 e_n }{\sqrt{1+\gamma_3^2}}\right)\bigg\rangle=
\frac{1}{1+\gamma_3^2}( g_{n-2,n-2} +2\gamma_3 g_{n-2,n} + \gamma_3^2 g_{n,n}).
\end{equation}
Hence by \eqref{stimaminore} and by \eqref{stimaultimacolonnan-2} we have that 
\begin{eqnarray}\label{stimagnn2}
|\gamma_3^2 g_{n,n}|=|F_2|\le  \|Tg\|,
\end{eqnarray}
where 
\[F_2= \big\langle g ({e_{n-2}+\gamma_3 e_n }), ({e_{n-2}+\gamma_1 e_n })\big\rangle - g_{n-2,n-2} -2\gamma_3 g_{n-2,n}.\]

Finally, we also know 
\begin{equation}
\bigg\langle g  \left(\frac{e_{n-2}+\gamma_2 e_n }{\sqrt{1+\gamma_2^2}}\right), \left(\frac{e_{n-2}+\gamma_2 e_n }{\sqrt{1+\gamma_2^2}}\right)\bigg\rangle =
\frac{1}{{(1+\gamma_2^2)}}\: \big( g_{n-1,n-1} +2\gamma_2 g_{n-1,n} + \gamma_2^2 g_{n,n} \big).
\end{equation}
Hence by \eqref{stimaminore} we have that 
\begin{eqnarray}\label{stimagnn}
|2\gamma_2 g_{n-1,n} + \gamma_2^2 g_{n,n} |=|F_3|\le  \|Tg\|,
\end{eqnarray}
where 
\[F_3 = \big\langle g \left({e_{n-2}+\gamma_2 e_n }\right), \left({e_{n-2}+\gamma_2 e_n }\right)\big\rangle - g_{n-1,n-1}.\]
Collecting together the above calculations lead to the linear system 
\[{\mathcal A} G= F,\]
where
$${\mathcal A} =\left(\begin{array}{ccc}2 & \gamma_1 \\0 & \gamma_3^2  \\ 2\gamma_2 & \gamma_2^2,
\end{array}\right)
$$
$G=(g_{n-1,n}\quad g_{n,n})^T$ and $F=(F_1 \:F_2\:F_3)^T$. If case \eqref{condizionea} holds, then we recover $G$ by inverting the square matrix 
$${\mathcal A_1} =\left(\begin{array}{ccc}2 & \gamma_1 \\0 & \gamma_3^2   
\end{array}\right).
$$
Otherwise, if case \eqref{condizioneb} holds, then we recover $G$ by inverting the square matrix 
$${\mathcal A_2} =\left(\begin{array}{ccc}2 & \gamma_1 \\2\gamma_2 & \gamma_2^2   
\end{array}\right),
$$
which concludes the proof. 
 \end{proof}


{\begin{proof}[Proof of theorem \ref{teorema principale}]
Let $y$, $\rho$ and $\Sigma$ be defined by \eqref{definizione Sigma}. From \eqref{g} we have
\begin{equation}\label{sigma i}
\sigma_i (y)= \left(\det g_i(y)\right)^{\frac{1}{2}} g_i^{-1}(y),\qquad\textnormal{for}\quad i=1,2,
\end{equation}
therefore
\begin{eqnarray}
\sigma_1(y) - \sigma_2(y) & = & \left[\left(\det g_1(y)\right)^{\frac{1}{2}} - \left(\det g_2(y)\right)^{\frac{1}{2}}\right]g_1^{-1}(y)\nonumber\\
& + & \left(\det g_2(y)\right)^{\frac{1}{2}}\left(g_1^{-1}(y) - g_2^{-1}(y)\right)\nonumber\\
& = & \frac{\det g_1(y) - \det g_2(y)}{\left(\det g_1(y)\right)^{\frac{1}{2}} + \left(\det g_2(y)\right)^{\frac{1}{2}}} \:g_1^{-1}(y)\nonumber\\
& + & \left(\det g_2(y)\right)^{\frac{1}{2}} g_1^{-1}(y)\left(g_2(y) - g_1(y)\right)g_2^{-1}(y),
\end{eqnarray}
which leads to
\begin{equation}\label{stima finale}
||\sigma_1(y) - \sigma_2(y)||_{{{\mathcal{L}({\mathbb{R}}^n, {\mathbb{R}}^n)}}}\leq C ||g_1(y) - g_2(y)||_{{{\mathcal{L}({\mathbb{R}}^n, {\mathbb{R}}^n)}}}.
\end{equation}
In view of Proposition \ref{lemma neumann map and K} and Proposition \ref{full stability metric} the proof is complete.
\end{proof}}

\section*{Acknowledgments}

RG was partly supported by Science Foundation Ireland under Grant number 16/RC/3918. The work of ES  was performed under the \text{PRIN grant No. 201758MTR2-007}. ES has also been supported by Gruppo Nazionale per l'Analisi \text{Matematica,} la Probabilità e le loro applicazioni (GNAMPA) by the grant "Problemi inversi per equazioni alle derivate parziali". 

\end{document}